\documentclass[12pt, reqno]{amsart}

\usepackage[utf8]{inputenc}
\usepackage[T1]{fontenc}
\usepackage{tikz,tikz-3dplot}
\usepackage{amsmath,amsfonts,amssymb,amsmath,latexsym,mathrsfs}
\usepackage{shuffle}
\usepackage[all]{xy}
\usepackage{stackengine}
 
\usepackage{enumerate}

\usepackage{hyperref}
 
\usepackage{tabu}
\usepackage{tikz-cd} 

\usepackage{comment}
\usepackage{url}
\usepackage{tikz}

\usepackage{xfrac}

\newtheorem{theorem}{Theorem}[section]

\newtheorem{proposition}[theorem]{Proposition}

\begin{document}

\title{Symmetric 2-cocycles with values in $\mathbb{C}^\times$}

\author{Mohamad \textsc{Maassarani} }
\maketitle
\begin{abstract} For many finite groups a symmetric $2$-cocycle $\alpha$ ($\alpha(g,h)=\alpha(h,g)$, for all pairs $(h,g)$ of the group) with values in $\mathbb{C}^\times$ is a coboundary. We show using a theoretic arguement and GAP that there is a group of order $64$ having a symmetric $2$-cocycle with a non trivial cohomology class.
\end{abstract}
\section{Introduction and main result}
We consider $2$-cocycles over finite groups with values in $\mathbb{C}^\times$. We will say that such cocycle $\alpha$ is symmetric if $\alpha(g,h)=\alpha(h,g)$ for all pairs of elements $(g,h)$ of the group. It is know that a symmetric $2$-cocycles with values in $\mathbb{C}^\times$ of an abelian group is a coboundary. For $G$ a finite group, the subgroup $H_S^2(G,\mathbb{C}^\times)$ of the Schur multiplier $H^2(G,\mathbb{C}^\times)$ consisting of classes of symmetric cocycles can be identified with a subgroup of the Bogomolov multiplier $B_0(G)$ (\cite{Cq}). In particular for a finite group with trivial Bogomolov multiplier all symmetric 2-cocycles with values in $\mathbb{C}^\times$ are coboundaries. According to \cite{Bog}, all simple groups has trivial Bogomolov multiplier and hence trivial $H_S^2(G,\mathbb{C}^\times)$. It is knonw that groups of order strictely less then $64$ have trivial Bogolomov Multiplier. The group $H_S^2(G,\mathbb{C}^\times)$ is also trivial for $G$ a Schur cover (see for instance \cite{Qrep},\cite{Cq}). In \cite{leb}, $\bar{C}$-groups are defined as groups defined by generators and relations and where the only relations are relations of the form $aba^{-1}=c$ and $d^k=1$ for $a,b,c,d$ generators. It is shown that for a $\bar{C}$-group $G$ there is an injective map between the enveloping group $A(G)$ of the conjugacy quandle of $G$ and $G\times \mathbb{Z}^{c_G}$ where $c_G$ is the number of conjugacy classe of $G$. In \cite{Cq}, we prove that for a finite group this map is injective if and only if $H_S^2(G,\mathbb{C}^\times)=0$. This implies that for a finite $\bar{C}$-group all symmetric $2$-cocycles with values in $\mathbb{C}^\times$ are coboundaries.\\\\
In \cite{Cq}, we proved that for $G$ finite $H_S^2(G,\mathbb{C}^\times)$ is zero if and only if the derived group of $G$ and the derived group of the enveloping group $A(G)$ of the conjugacy quandle of $G$ are isomorphic. We show that for $G$ corresponding to the group $SmallGroup(64,149)$ of order $64$ of GAP, the derived group of $A(G)$ has more elements then the derived group of $G$, proving that $H^2_S(G,\mathbb{C}^\times)$ is non trivial and hence that $SmallGroup(64,149)$ posses a symmetric $2$-cocycle with values in $\mathbb{C}^\times$ with a non trivial cohomology class (not a coboundary). To establish the result we use GAP 4.15.1. The code executes in seconds on a standard Hardware. For instance, the existence of such cocycle proves that the conujugacy quandle of $SmallGroup(64,149)$ posses an irreducible quandle representation that can't be obtained as product of a quandle character and a linear representation of the group \cite{Cq}. 
 \section{Existence of the $2$-cocyle}
Let $G$ be a group. The enveloping group of the conjugacy quandle of $G$ is the group :
$$ A(G)=\langle e_g,\: g\in G \vert e_ge_he_g^{-1}=e_{ghg^{-1}} \: \text{for} \: g,h\in G \rangle.$$

As in the introduction, we will denote $H_S^2(G,\mathbb{C}^\times)$ the subgroup of the Schur multiplier $H^2(G,\mathbb{C}^\times)$ consisting of classes of symmetric cocycles. The following is a part of proposition $3.4$ of \cite{Cq} : 
\begin{proposition}\label{cor}
For $G$ finite the derived subgroup of $A(G)$ is isomorphic to the derived subgroup of $G$ if and only if $H^2_S(G,\mathbb{C}^\times)=0$.
\end{proposition}
 
From now on, $G$ \textbf{is the group corresponding to} $SmallGroup(64,149)$ of GAP.  
\begin{proposition}
The derived group of $G$ has order $8$.
\end{proposition}
\begin{proof}
We use the code :\\\\
G := SmallGroup(64,149);\\
dG:=DerivedSubgroup(G);\\
a:=Size(dG);\\\\
GAP returns $8$ for the value of $a$.
\end{proof}
We implement the group $A(G)$ in GAP by adding to the code in the previous proposition the lines :\\\\
elms := Elements(G);\\
r := Length(elms);\\
F := FreeGroup(r); \\
gens := GeneratorsOfGroup(F);\\
rels := [\:];\\
for i in [1..r] do\\
for j in [1..r] do\\
        h := elms[j];\\
        po:=elms[i]*elms[j]*elms[i] $\:\widehat{\:}\:$-1 ;\\
        k := Position(elms, po);\\
        Add(rels, gens[i] * gens[j] * gens[i]$\:\widehat{\:}\:$-1*gens[k]$\:\widehat{\:}\:$-1);\\
	    od;\\
od;\\
H := F / rels;\\\\
Hence the group $A(G)$ correspond to the group $H$ in the code.
\begin{proposition}
The derived group of $A(G)$ has order greater then $16$.
\end{proposition}
\begin{proof} 
We find using the following code lines :\\\\
f := EpimorphismPGroup(H,2,3);\\
K:= Image(f);\\
dK:=DerivedSubgroup(K);\\
n:=Size(dK);\\\\
that there is an eprimorphism from $A(G)$ to a $2$-group  $K$ of class $3$ and that the derived subgroup of $K$ is of order $n=16$. Since the derived group of a quotient of $A(G)$ has order $16$, the derived group of $A(G)$ has order greater then $16$
\end{proof}
 
\begin{proposition}
For $G$ corresponding to $SmallGroup(64,149)$ of GAP, $H^2_S(G,\mathbb{C}^\times)\neq 0$.
\end{proposition}
\begin{proof}
This follows from proposition \ref{cor} and the last two propositions.
\end{proof}
 The author has no prior experience with GAP. He used the help of Google's AI.

 \end{document}